\documentclass[12pt]{amsart}

\usepackage[english]{babel}
\usepackage{mathrsfs,amssymb}
\usepackage{mathtools}
\usepackage[colorlinks, citecolor = blue]{hyperref}

\usepackage[shortlabels]{enumitem}
\setlist[itemize]{leftmargin=25pt}
\setlist[enumerate]{leftmargin=25pt}

\newtheorem{theorem}{Theorem}[section]
\newtheorem{lemma}[theorem]{Lemma}
\newtheorem{prop}[theorem]{Proposition}
\newtheorem{cor}[theorem]{Corollary}

\theoremstyle{definition}

\newtheorem{con}[theorem]{Conjecture}

\theoremstyle{remark}

\numberwithin{equation}{section}
\usepackage[colorinlistoftodos,prependcaption,textsize=small]{todonotes}

\let \la=\lambda

\let \d=\delta

\let \a=\alpha
\let \f=\varphi
\let \b=\beta

\allowdisplaybreaks

\begin{document}
\title[Maximal operator on Banach function spaces]
{A note on the maximal operator on Banach function spaces}

\author[A.K. Lerner]{Andrei K. Lerner}
\address[A.K. Lerner]{Department of Mathematics,
Bar-Ilan University, 5290002 Ramat Gan, Israel}
\email{lernera@math.biu.ac.il}

\thanks{The author was supported by ISF grant no. 1035/21.}

\begin{abstract}
In this note we answer positively to two conjectures proposed by Nieraeth \cite{N23} about
the maximal operator on rescaled Banach function spaces. We also obtain a new criterion
saying when the maximal operator bounded on a Banach function space $X$ is also bounded
on the associate space $X'$.
\end{abstract}



\maketitle

\section{Introduction}
Let $X$ be a Banach function space over ${\mathbb R}^n$. Denote by $X'$ its associate space. Next, for $p>0$ denote by $X^p$ the space with finite semi-norm
$$\|f\|_{X^p}:=\||f|^{1/p}\|_{X}^p.$$

Let $s>r\ge 1$. Assume that $X$ is $r$-convex and $s$-concave. Define the $(r,s)$-rescaled Banach function space of $X$ as
$$X_{r,s}:=\Big[\big[(X^r)'\big]^{(\frac{s}{r})'}\Big]'.$$
This space was introduced in a recent work by Nieraeth \cite{N23}. The factorization formula (see \cite[Cor. 2.12]{N23})
$$X=(X_{r,s})^{\frac{1}{r}-\frac{1}{s}}\cdot L^s({\mathbb R}^n)$$
makes the space $X_{r,s}$ an important tool in the extrapolation theory for general Banach function spaces (see \cite{N23} and also \cite{LN23}).

Let $M$ be the Hardy--Littlewood maximal operator defined by
$$Mf(x):=\sup_{Q\ni x}\frac{1}{|Q|}\int_Q|f|,$$
where the supremum is taken over all cubes $Q\subset {\mathbb R}^n$ containing the point $x$.

In \cite[Conjectures 2.38, 2.39]{N23} the following conjectures were stated.

\begin{con}\label{con1} {\it Let $s>r\ge 1$ and let $X$ be an $r$-convex and $s$-concave Banach function space over ${\mathbb R}^n$.
Suppose that $M$ is bounded on
$\big[(X^r)'\big]^{(\frac{s}{r})'}$. Then the following are equivalent:
\begin{enumerate}[(i)]
\item
$M$ is bounded on $X_{r,s}$;
\item $M$ is bounded on $X^r$.
\end{enumerate}}
\end{con}

\begin{con}\label{con2} {\it Let $s>r\ge 1$ and let $X$ be an $r$-convex and $s$-concave Banach function space over ${\mathbb R}^n$. Then the following are equivalent:
\begin{enumerate}[(i)]
\item
$M$ is bounded on $X_{r,s}$ and on $(X_{r,s})'$;
\item $M$ is bounded on $X^r$ and $(X')^{s'}$.
\end{enumerate}}
\end{con}

Observe that actually both above conjectures are formulated in \cite{N23} in a more general setting of quasi-Banach function spaces and abstract maximal operators.
We restrict ourselves to Banach function spaces and the most standard maximal operator. Note also that the implications $({\rm i})\Rightarrow ({\rm ii})$ were shown in \cite{N23}
for both conjectures and, hence, the question is about the converse implications $({\rm ii})\Rightarrow ({\rm i})$.

In this note we show that both Conjectures \ref{con1} and \ref{con2} are true. A useful tool in our proofs will be a new criterion about the interplay between the
boundedness of $M$ on $X$ and $X'$, which is perhaps of some independent interest. This criterion is formulated in terms of the local maximal operator $m_{\la}$
acting on measurable functions on ${\mathbb R}^n$ by
$$m_{\la}f(x):=\sup_{Q\ni x}(f\chi_Q)^*(\la|Q|),\quad \la\in (0,1),$$
where $f^*$ stands for the standard non-increasing rearrangement of $f$, and the supremum is taken over all cubes $Q\subset {\mathbb R}^n$
containing the point~$x$.

Given a Banach function space $X$, we associate with it the function~$\f_X$ defined by
$$\f_X(\la):=\inf_{\|f\|_{X}=1}\|m_{\la}f\|_{X},\quad \la\in (0,1).$$
Observe that since $m_{\la}f\ge |f|$ a.e. (see, e.g., \cite[Lemma 6]{L04}), we have $\f_X(\la)\ge 1$ for all $\la\in (0,1)$.
Also, it is immediate that the function $\f_X(\la)$ is non-increasing.

We have the following result.

\begin{theorem}\label{mr} Let $X$ be a Banach function space over ${\mathbb R}^n$, and assume that the maximal operator $M$ is bounded on $X$. Then
\begin{enumerate}[(i)]
\item the function $\f_{X'}(\la)$ is unbounded;
\item if there exists $\la_0\in (0,1)$ such that $\f_X(\la_0)>1$, then $M$ is bounded on $X'$.
\end{enumerate}
\end{theorem}

An immediate consequence of Theorem \ref{mr} is the following criterion.

\begin{theorem}\label{cor1} Let $X$ be a Banach function space over ${\mathbb R}^n$, and assume that the maximal operator $M$ is bounded on $X$.
The following statements are equivalent:
\begin{enumerate}[(i)]
\item
$M$ is bounded on $X'$;
\item
the function $\f_X(\la)$ is unbounded;
\item
there exists $\la_0\in (0,1)$ such that $\f_X(\la_0)>1$.
\end{enumerate}
\end{theorem}

Indeed, observe that the implication $({\rm i})\Rightarrow ({\rm ii})$ follows by taking $X'$ instead of $X$ in Theorem \ref{mr} and taking into account that, by the Lorentz--Luxembourg theorem, $X''=X$. Next, $({\rm ii})\Rightarrow ({\rm iii})$ is trivial. Finally, $({\rm iii})\Rightarrow ({\rm i})$ is contained in item (ii) of Theorem \ref{mr}.

The equivalence $({\rm ii})\Leftrightarrow ({\rm iii})$ in Theorem \ref{cor1} can be written in the following form.

\begin{cor}\label{cor2} Let $X$ be a Banach function space over ${\mathbb R}^n$, and assume that the maximal operator $M$ is bounded on $X$.
Then either $\f_{X}(\la)$ is unbounded or $\f_{X}(\la)\equiv 1$ for all $\la\in (0,1)$.
\end{cor}

The function $\f_X$ is especially useful when dealing with scaled spaces~$X^p$. Indeed, using that $(|f|^r)^*(t)=f^*(t)^r, r>0$, we obtain
$$
\f_{X^p}(\la)=\inf_{\||f|^{1/p}\|_X=1}\|m_{\la}(|f|^{1/p})\|_X^p=\f_X(\la)^p, \quad p>0.
$$
Thus, if $\f_{X^{p_0}}$ is unbounded for some $p_0>0$, then $\f_{X^p}$ is unbounded for all $p>0$. This fact combined with the above results will be crucial in proving Conjectures \ref{con1} and \ref{con2}.

In order to prove Conjecture \ref{con2}, we will also essentially use the notion of $A_p$-regularity of Banach function spaces. This notion was considered by Rutsky \cite{R15, R14}.

The paper is organized as follows. Section 2 contains some preliminaries. In Section 3 we prove Theorem \ref{mr}. Conjectures \ref{con1} and \ref{con2} are proved in Section 4.
We will also give an alternative proof, based on the function $\f_X$, of a recent result by Lorist and Nieraeth \cite{LN23} about the boundedness of $M$ on $X$ and $X'$.

\section{Preliminaries}
\subsection{Banach function spaces}
Let $L^0({\mathbb R}^n)$ denote the space of measurable functions on ${\mathbb R}^n$. A vector space $X\subseteq L^0({\mathbb R}^n)$ equipped with a norm $\|\cdot\|_X$
is called a Banach function space over ${\mathbb R}^n$ if it satisfies the following properties:
\begin{itemize}
\item {\it Ideal property}: If $f\in X$ and $g\in L^0({\mathbb R}^n)$ with $|g|\le |f|$, then $g\in X$ and $\|g\|_X\le \|f\|_{X}$.
\item {\it Fatou property}: If $0\le f_j\uparrow f$ for $\{f_j\}$ in $X$ and $\sup_j\|f_j\|_{X}<\infty$, then $f\in X$ and $\|f\|_{X}=\sup_j\|f_j\|_{X}$.
\item {\it Saturation property}: For every measurable set $E\subset {\mathbb R}^n$ of positive measure, there exists a measurable subset $F\subseteq E$ of positive measure
such that $\chi_F\in X$.
\end{itemize}

We refer to a recent survey by Lorist and Nieraeth \cite{LN24} about (quasi)-Banach function spaces, where, in particular, one can find a discussion about the above choice of axioms.

The following statement is an equivalent formulation of the Fatou property (see, e.g., \cite[Lemma 3.5]{LN24}).

\begin{prop}\label{fat} Let $X$ be a Banach function space on ${\mathbb R}^n$. Then for every sequence $f_j\in X$,
$$\|\liminf_{j\to \infty}f_j\|_{X}\le \liminf_{j\to \infty}\|f_j\|_{X}.$$
\end{prop}

The next statement is also well known.

\begin{prop}\label{mx} Let $X$ be a Banach function space on ${\mathbb R}^n$, and assume that $M$ is bounded on $X$. Then $\chi_Q\in X$ for every cube $Q\subset {\mathbb R}^n$.
\end{prop}

\begin{proof} Fix a cube $Q$. By saturation property, there is a set $F\subseteq Q$ of positive measure such that $\chi_F\in X$. Since $M$ is bounded on $X$,
$$\frac{|F|}{|Q|}\|\chi_Q\|_{X}\le \|M\chi_F\|_{X}\le c\|\chi_F\|_{X},$$
which implies $\|\chi_Q\|_{X}<\infty$.
\end{proof}

Given a Banach function space $X$, we define the associate space (also called the K\"othe dual) $X'$ as the space of all $f\in L^0({\mathbb R}^n)$ such that
$$\|f\|_{X'}:=\sup_{\|g\|_{X}\le 1}\int_{{\mathbb R}^n}|fg|<\infty.$$
By the Lorentz--Luxembourg theorem (see \cite[Th. 71.1]{Z67}), we have $X''=X$ with equal norms.

Let $X$ be a Banach function space, and let $1\le p, q\le \infty$. We say that $X$ is $p$-convex if
$$\|(|f|^p+|g|^p)^{1/p}\|_{X}\le (\|f\|_{X}^p+\|g\|_{X}^p)^{1/p},\quad f,g\in X,$$
and we say that $X$ is $q$-concave if
$$(\|f\|_{X}^q+\|g\|_{X}^q)^{1/q}\le \|(|f|^q+|g|^q)^{1/q}\|_{X} ,\quad f,g\in X.$$

\subsection{$A_p$ weights}
By a weight we mean a non-negative locally integrable function on~${\mathbb R}^n$. Given a weight $w$ and a measurable set $E\subset {\mathbb R}^n$,
denote $w(E):=\int_Ew$ and $\langle w\rangle_E:=\frac{1}{|E|}\int_w$.

Recall that a weight $w$ satisfies the $A_1$ condition if
$$[w]_{A_1}:=\Big\|\frac{Mw}{w}\Big\|_{L^{\infty}}<\infty;$$
a weight $w$ satisfies the $A_p, 1<p<\infty,$ condition if
$$[w]_{A_p}:=\sup_{Q}\langle w\rangle_Q\langle w^{-p'/p}\rangle_Q^{p/p'}<\infty;$$
a weight $w$ satisfies the $A_{\infty}$ condition if
$$[w]_{A_{\infty}}:=\sup_Q\frac{\int_{Q}M(w\chi_Q)}{w(Q)}<\infty.$$

It was shown in \cite{HP13} that if $w\in A_{\infty}$, then for $r:=1+\frac{1}{c_n[w]_{A_{\infty}}}$ and for every cube $Q$,
$$\Big(\frac{1}{|Q|}\int_Qw^r\Big)^{1/r}\le 2\frac{1}{|Q|}\int_Qw.$$
From this, by H\"older's inequality we obtain that for every cube $Q$ and any measurable subset $E\subset Q$,
\begin{equation}\label{ainf}
w(E)\le 2\Big(\frac{|E|}{|Q|}\Big)^{\d}w(Q),
\end{equation}
where $\d:=\frac{1}{r'}=\frac{1}{1+c_n[w]_{A_{\infty}}}$.

\subsection{$A_p$-regular Banach function spaces}
Let $X$ be a Banach function space, and let $1\le p\le \infty$. We say that $X$ is $A_{p}$-regular if there exist $C_1, C_2>0$ such that for every $f\in X$ there is an $A_{p}$ weight $w\ge |f|$ a.e. with $[w]_{A_{p}}\le C_1$
and $\|w\|_{X}\le C_2\|f\|_{X}$.

\begin{prop}\label{ela1}
A Banach function space $X$ is $A_1$-regular if and only if $M$ is bounded on $X$.
\end{prop}

\begin{proof} Indeed, one direction is trivial, namely, if $X$ is $A_1$-regular, then
$$\|Mf\|_X\le \|Mw\|_{X}\le C_1\|w\|_{X}\le C_1C_2\|f\|_{X}.$$

Conversely, if $M$ is bounded on $X$, then there exists $r>1$ depending only on $\|M\|_{X\to X}$ such that $M_r$ is also bounded on $X$, where $M_rf:=M(|f|^r)^{1/r}$ (see, e.g., \cite{LO10} for the proof of this result).
Combining this with the well known fact that $M_rf\in A_1$ \cite{CR80} with the $A_1$-constant depending only on $r$ and $n$, we obtain that $X$ is $A_1$-regular.
\end{proof}

The following result is an abridged version of the characterization obtained by Rutsky \cite[Th. 2]{R15}.

\begin{theorem}[\cite{R15}]\label{rch} Let $X$ be a Banach function space, and let $1<p<\infty$. The following statements are equivalent:
\begin{enumerate}[(i)]
\item both $X^{1/p}$ and $(X^{1/p})'$ are $A_1$-regular;
\item $X'$ is $A_p$-regular.
\end{enumerate}
\end{theorem}

A difficult part of this result is the implication $({\rm i})\Rightarrow({\rm ii})$. We outline a slightly different proof for the sake of completeness. Essentially this is contained in the following
result by Rubio de Francia \cite[Section 3]{RDF87}.

\begin{theorem}[\cite{RDF87}]\label{rdf} Let $X$ be a Banach function space, and let $p\in (1,\infty)$. Assume that $T$ is a linear operator bounded on $X^{1/p}(\ell^p)$, namely,
there exists $C>0$ such that for every sequence $\{f_j\}$,
\begin{equation}\label{mcond}
\|(\sum_{j}|Tf_j|^p)^{1/p}\|_{X^{1/p}}\le C\|(\sum_j|f_j|^p)^{1/p}\|_{X^{1/p}}.
\end{equation}
Then for every $f\in X'$, which is positive almost everywhere, there exists a function $w\ge f$ such that $\|w\|_{X'}\le 2\|f\|_{X'}$ and $T$ is bounded on $L^p(w)$ with the operator
norm $\|T\|_{L^p(w)\to L^p(w)}$ depending only on $C$ in (\ref{mcond}).
\end{theorem}

Now observe that by Proposition \ref{ela1}, condition (i) of Theorem \ref{rch} is equivalent to $M$ being bounded on $X^{1/p}$ and on $(X^{1/p})'$. In turn, by \cite[Th. 4.9]{N23}, this implies
that (\ref{mcond}) holds for every standard Calder\'on--Zygmund operator $T$. It remains to choose any nondegenerate  Calder\'on--Zygmund operator $T$, namely, an operator $T$ for which $\|T\|_{L^p(w)\to L^p(w)}<\infty$
implies $w\in A_p$, and an application of Theorem~\ref{rdf} completes the proof.

The following result was also obtained by Rutsky \cite[Prop. 7]{R14}.

\begin{prop}[\cite{R14}]\label{propr} Let $X$ be a Banach function space over ${\mathbb R}^n$ such that $X$ is $A_p$-regular for some $1\le p<\infty$. Suppose that there exists $\d>0$ such
that $M$ is bounded on $X^{\d}$. Then $M$ is bounded on $X$.
\end{prop}

We sketch the proof of this statement. By the well known property of the $A_p$ weights \cite{CF74},
if $w\in A_p$, then there exist $\a,\b>0$ depending only on $[w]_{A_p}$ and $p$ such that for every cube $Q$,
$$\a|Q|<|\{x\in Q:w(x)>\b\langle w\rangle_Q\}|.$$
From this, by Chebyshev's inequality, for every $\d>0$,
$$Mw(x)\le C_{\d,\a,\b}M_{\d}w(x).$$
Hence, taking $f\in X$ and the corresponding $w\in A_p$ from the definition of the $A_p$-regularity, we obtain
\begin{eqnarray*}
\|Mf\|_X&\le& \|Mw\|_{X}\le C\|M_{\d}w\|_{X}\\
&=&C\|M(w^{\d})\|_{X^{\d}}^{1/\d}\le C\|w\|_{X}\le C\|f\|_{X},
\end{eqnarray*}
which proves the result.

\subsection{Rearrangements and the maximal operator $m_{\la}$}
Denote by $S_0({\mathbb R}^n)$ the space of measurable functions $f$ on ${\mathbb R}^n$ such that
$$\mu_{f}(\a):=|\{x\in {\mathbb R}^n:|f(x)|>\a\}|<\infty$$
for any $\a>0$. Recall that for $f\in S_0({\mathbb R}^n)$ the non-increasing rearrangement $f^*$ is defined by
$$f^*(t):=\inf\{\a>0: \mu_{f}(\a)\le t\}, \quad t>0.$$

It can be easily seen from the definition of the rearrangement that for every measurable function $f$ and any cube $Q$,
$$(f\chi_Q)^*(\la|Q|)>\a\Leftrightarrow |Q\cap\{|f|>\a\}|>\la|Q|,\quad \a>0,\la\in (0,1).$$
From this we have that
\begin{equation}\label{prml1}
\{x\in {\mathbb R}^n:m_{\la}f(x)>\a\}=\{x\in {\mathbb R}^n:M\chi_{\{|f|>\a\}}(x)>\la\},
\end{equation}
and, for every measurable set $E\subset {\mathbb R}^n$,
\begin{equation}\label{prml2}
m_{\la}(\chi_E)(x)=\chi_{\{M\chi_E>\la\}}(x).
\end{equation}

\section{Proof of Theorem \ref{mr}}
We start with part (i) of Theorem \ref{mr}, and we will show that it even holds under a weaker assumption, namely, instead of assuming that $M$ is bounded on $X$ (which, by Proposition \ref{ela1},
says that $X$ is $A_1$-regular) it suffices to assume the $A_{\infty}$-regularity of $X$. This is a simple corollary of the following lemma.

\begin{lemma}\label{pr} Let $w\in A_{\infty}$. Then for all $f\in L^1(w)$ and $\la\in (0,1)$,
\begin{equation}\label{ineqml}
\int_{{\mathbb R}^n}|f|w\le 2^{n+1}\la^{\d}\int_{{\mathbb R}^n}(m_{\la}f)w,
\end{equation}
where $\d:=\frac{1}{1+c_n[w]_{A_{\infty}}}$.
\end{lemma}

\begin{proof}
Observe that (\ref{ineqml}) is equivalent to the same inequality but with $f=\chi_E$,
where $E$ is an arbitrary measurable set of finite measure. Indeed, assume that (\ref{ineqml}) is true for $f=\chi_E$. By (\ref{prml2}), this means that
\begin{equation}\label{int}
w(E)\le 2^{n+1}\la^{\d}w\{x: M\chi_E>\la\}.
\end{equation}
Taking here $E:=\{x:|f|>\a\}$ and applying (\ref{prml1}), we obtain
$$w(\{x:|f|>\a\})\le 2^{n+1}\la^{\d}w(\{x:m_{\la}f>\a\}),$$
which, in turn, implies (\ref{ineqml}) by integrating over $\a\in (0,\infty)$ (in order to ensure that the set $\{x:|f|>\a\}$ is of finite measure, one can assume first that $f$ is compactly supported and then use the monotone convergence theorem).

Hence, it suffices to prove (\ref{int}). Let $M^d$ denote the dyadic maximal operator. By the Calder\'on--Zygmund decomposition, the set $\{x:M^d\chi_E>\la\}$ can be written as the union of pairwise disjoint cubes $Q_j$ satisfying
$$\la<\frac{|Q_j\cap E|}{|Q_j|}\le 2^n\la.$$
From this, by (\ref{ainf}),
$$w(Q_j\cap E)\le 2^{1+n\d}\la^{\d}w(Q_j)\le 2^{n+1}\la^{\d}w(Q_j).$$
Summing up this inequality yields (\ref{int}), and therefore the proof is complete.
\end{proof}

\begin{proof}[Proof of Theorem \ref{mr}, part (i)]
We will show that if $X$ is $A_{\infty}$-regular, then there exist $C,\d>0$ such that $\f_{X'}(\la)\ge C\la^{-\d}$ for all $\la\in (0,1)$.

Given $g\in X$, take an $A_{\infty}$ weight $w$ such that $|g|\le w$ a.e., where $[w]_{A_{\infty}}\le C_1$ and $\|w\|_{X}\le C_2\|g\|_{X}$.
By Lemma \ref{pr}, there exists $\d>0$ (one can take $\d:=\frac{1}{1+c_nC_1}$) such that
\begin{eqnarray*}
\int_{{\mathbb R}^n}|fg|&\le& \int_{{\mathbb R}^n}|f|w\le 2^{n+1}\la^{\d}\int_{{\mathbb R}^n}(m_{\la}f)w\\
&\le& 2^{n+1}\la^{\d}\|m_{\la}f\|_{X'}\|w\|_{X}\le 2^{n+1}C_2\la^{\d}\|m_{\la}f\|_{X'}\|g\|_{X}.
\end{eqnarray*}
From this, taking the supremum over all $g\in X$ with $\|g\|_{X}=1$ yields
$$\|f\|_{X'}\le 2^{n+1}C_2\la^{\d}\|m_{\la}f\|_{X'}.$$
Hence, $\f_{X'}(\la)\ge \frac{1}{2^{n+1}C_2}\la^{-\d}$, and the proof is complete.
\end{proof}

Turn to part (ii) of Theorem \ref{mr}. This part is a simple combination of several known results, which we will describe below.

Define the Fefferman--Stein sharp function $f^{\#}$ by
$$f^{\#}(x):=\sup_{Q\ni x}\frac{1}{|Q|}\int_Q|f-\langle f\rangle_Q|,$$
where the supremum is taken over all cubes $Q$ containing the point $x$.
It was proved by Fefferman and Stein \cite{FS72} that for every $p>1$ and for all $f\in S_0({\mathbb R}^n)$,
$$\|f\|_{L^p}\le C_{n,p}\|f^{\#}\|_{L^p}.$$
Having this result in mind, we say that a Banach function space $X$ has the Fefferman--Stein property if there exists $C>0$ such that
$$\|f\|_{X}\le C\|f^{\#}\|_{X}$$
for all $f\in S_0({\mathbb R}^n)$.

The following characterization was obtained in \cite[Cor. 4.3]{L10}.

\begin{theorem}[\cite{L10}]\label{fsch} Let $X$ be a Banach function space over ${\mathbb R}^n$, and assume that $M$ is bounded on $X$. Then $M$ is bounded on $X'$ if and only if $X$ has
the Fefferman--Stein property.
\end{theorem}

Further, we will use the following pointwise estimate obtained in \cite[Th. 2]{L00}.

\begin{theorem}[\cite{L00}]\label{mla} For any locally integrable function $f$ and for all $x\in {\mathbb R}^n$,
$$m_{\la}(Mf)(x)\le C_{n,\la}f^{\#}(x)+Mf(x).$$
\end{theorem}

Now, the second part of Theorem \ref{mr}, modulo some technicalities, is just a combination of two above results.

\begin{proof}[Proof of Theorem \ref{mr}, part (ii)] The proof is almost identical to the proof of a similar result proved in \cite[Th. 4.1]{L10}.

Our goal is to show that if $\f_X(\la_0)>1$ for some $\la_0\in (0,1)$, then $X$ has the Fefferman--Stein property. Thus, by Theorem \ref{fsch}, we would
obtain that $M$ is bounded on $X'$.

Since $(|f|)^{\#}(x)\le 2f^{\#}(x)$ (see, e.g., \cite[p. 155]{G14}), we will assume that $f\ge 0$. By Theorem \ref{mla},
$$\f_{X}(\la_0)\|Mf\|_{X}\le \|m_{\la_0}(Mf)\|_{X}\le C_{n,\la_0}\|f^{\#}\|_{X}+\|Mf\|_{X}.$$
Assuming that $f\in X$, and using that $M$ is bounded on $X$, we obtain that $\|Mf\|_{X}<\infty$.
Therefore, by the above estimate,
\begin{equation}\label{onx}
\|f\|_{X}\le \|Mf\|_{X}\le \frac{C_{n,\la_0}}{\f_X(\la_0)-1}\|f^{\#}\|_{X}.
\end{equation}

Now, in order to show that $X$ has the Fefferman--Stein property, it remains to extend (\ref{onx}) from $f\in X$ to $f\in S_0({\mathbb R}^n)$.
Assume first that $f\in S_{0}({\mathbb R}^n)\cap L^{\infty}$. We will use the fact proved in \cite[Lemma 4.5]{L10} and saying that there is a sequence $\{f_j\}$
of bounded and compactly supported functions such that $f_j\to f$ a.e. and $(f_j)^{\#}(x)\le c_nf^{\#}(x)$.

Observe that, by Proposition \ref{mx}, each $f_j$ belongs to $X$. Therefore, by (\ref{onx}),
$$\|f_j\|_{X}\le C\|(f_j)^{\#}\|_{X}\le C'\|f^{\#}\|_{X}.$$
From this, by Proposition \ref{fat},
$$\|f\|_{X}\le C'\|f^{\#}\|_{X}.$$

It remains to extend this estimate from $f\in S_{0}({\mathbb R}^n)\cap L^{\infty}$ to $f\in S_{0}({\mathbb R}^n)$.
For $f\in S_{0}({\mathbb R}^n)$ and $N>0$ define $f_N:=\min(f,N)$. Then $f_N\in S_{0}({\mathbb R}^n)\cap L^{\infty}$. Using that $(f_N)^{\#}(x)\le \frac{3}{2}f^{\#}(x)$ (see \cite[p. 155]{G14})
and applying the previous estimate, we obtain
$$\|f_N\|_{X}\le C\|f^{\#}\|_{X}.$$
Applying again Proposition \ref{fat} proves the Fefferman--Stein property of~$X$, and therefore the proof is complete.
\end{proof}

\section{Proof of Conjectures \ref{con1} and \ref{con2}}
We start with the following statement which collects some standard properties related to the boundedness of $M$ on a Banach function space~$X$.

\begin{prop}\label{elpr} Let $X$ be a Banach function space, and assume that~$M$ is bounded on $X$. Then
\begin{enumerate}[(i)]
\item $M$ is bounded on $X^r$ for all $r\in (0,1)$;
\item there exists $r>1$ such that $M$ is bounded on $X^r$;
\item $M$ is bounded on $(X')^{r}$ for all $r\in (0,1)$.
\end{enumerate}
\end{prop}

\begin{proof} Observe that $M$ is bounded on $X^r$ if and only if $M_r$ is bounded on $X$. Therefore, part (i) follows trivially by H\"older's inequality.
In turn, part (ii) is just a reformulation of the result \cite{LO10} saying that if $M$ is bounded on $X$, then there exists $r>1$ depending only on $\|M\|_{X\to X}$ such
that $M_r$ is also bounded on $X$.

Part (iii) is an immediate consequence of the Fefferman--Stein inequality \cite{FS71} saying that for all locally integrable $f,g$ and for all $p>1$,
$$\int_{{\mathbb R}^n}(Mf)^p|g|dx\le C_{n,p}\int_{{\mathbb R}^n}|f|^p(Mg)dx.$$
From this
$$\int_{{\mathbb R}^n}(Mf)^p|g|dx\le C_{n,p}\||f|^p\|_{X'}\|Mg\|_{X}\le C\||f|^p\|_{X'}\|g\|_{X}.$$
Hence, by duality, for all $p>1$,
$$\|(Mf)^p\|_{X'}\le C\||f|^p\|_{X'},$$
which finishes the proof.
\end{proof}

As we mentioned in the Introduction, the implications $({\rm i})\Rightarrow ({\rm ii})$ are known for both Conjectures \ref{con1} and \ref{con2}. For the sake of the completeness we give different proofs based on Theorem \ref{mr}.

\begin{proof}[Proof of Conjecture \ref{con1}]
Observe that since $X$ is $r$-convex and $s$-concave, the space $X_{r,s}$ is a Banach function space (see \cite{N23}). Therefore, by the Lorentz--Luxembourg theorem,
$\big[(X^r)'\big]^{(\frac{s}{r})'}=X_{r,s}'$.

Let us start with the implication $({\rm i})\Rightarrow ({\rm ii})$. Since $M$ is bounded on $X_{r,s}$, by Theorem \ref{mr}, the function $\f_{\big[(X^r)'\big]^{(\frac{s}{r})'}}$ is unbounded. Hence, $\f_{(X^r)'}$ is unbounded
as well. Next, since $M$ is bounded on $\big[(X^r)'\big]^{(\frac{s}{r})'}$, by the first part of Proposition \ref{elpr}, $M$ is bounded on $(X^r)'$. This along with unboundedness of $\f_{(X^r)'}$ implies,
by Theorem \ref{cor1}, that $M$ is bounded on $(X^r)''=X^r$.

Turn to $({\rm ii})\Rightarrow ({\rm i})$. Since $M$ is bounded on $X^r$, by Theorem \ref{mr}, $\f_{(X^r)'}$ is unbounded. Hence, $\f_{\big[(X^r)'\big]^{(\frac{s}{r})'}}$ is unbounded as well. This coupled with the
boundedness of $M$ on $\big[(X^r)'\big]^{(\frac{s}{r})'}$ implies, by Theorem \ref{cor1}, that $M$ is bounded on $X_{r,s}$.
\end{proof}

\begin{proof}[Proof of Conjecture \ref{con2}] Let us start with the implication $({\rm i})\Rightarrow ({\rm ii})$. In the previous proof we showed that if $M$ is bounded on $X_{r,s}$ and $X_{r,s}'$,
then $M$ is bounded on $X^r$. Using the fact that $X_{r,s}'=(X')_{s',r'}$ \cite[Pr.~2.14]{N23}, in a similar way we obtain that $M$ is bounded on $(X')^{s'}$.

Turn to $({\rm ii})\Rightarrow ({\rm i})$. Since $M$ is bounded on $(X')^{s'}$, by the first part of Proposition \ref{elpr}, $M$ is bounded on $X'$. Hence, by Theorem \ref{mr}, $\f_X$ is unbounded, and
so $\f_{X^r}$ is unbounded as well. This coupled with the boundedness of $M$ on $X^r$ implies, by Theorem \ref{cor1}, that $M$ is bounded on $(X^r)'$.
In a similar way we obtain that $M$ is bounded on $[(X')^{s'}]'$.

Setting $Y:=\big[(X^r)'\big]^{(\frac{s}{r})'}$ and $q:=(\frac{s}{r})'$, we obtain that $Y^{1/q}$ and $(Y^{1/q})'$ are $A_1$-regular. Therefore, by Theorem \ref{rch}, $Y'=X_{r,s}$ is $A_q$-regular.
Further, it was shown in the proof of \cite[Pr. 2.14]{N23} that $X_{r,s}^{\theta}=[(X')^{s'}]'$ for $\theta:=\frac{1}{(r'/s')'}$. Hence, $M$ is bounded on $X_{r,s}^{\theta}$, which, along with the $A_q$-regularity
of $X_{r,s}$, proves, by Proposition \ref{propr}, that $M$ is bounded on $X_{r,s}$. By symmetry, using that $X_{r,s}'=(X')_{s',r'}$ and applying the same argument, we obtain that $M$ is bounded on $X_{r,s}'$,
and, therefore, the proof is complete.
\end{proof}

In order to further illustrate the method based on Theorem \ref{mr}, we give an alternative proof of the following recent result of Lorist and Nieraeth \cite{LN23}.

\begin{theorem}[\cite{LN23}] Let $r^*\in (1,\infty)$ and let $X$ be an $r^*$-convex Banach function space over ${\mathbb R}^n$.
Then the following are equivalent:
\begin{enumerate}[(i)]
\item
We have $M:X\to X$ and $M:X'\to X'$;
\item There is an $r_0\in (1,r^*]$ so that for all $r\in (1,r_0)$ we have
$$M:X^r\to X^r,\quad M:(X^r)'\to (X^r)';$$
\item
There is an $r\in (1,r^*]$ so that $M:(X^r)'\to (X^r)'$.
\end{enumerate}
\end{theorem}

\begin{proof}
We start with $({\rm i})\Rightarrow ({\rm ii})$. By a combination of the first two parts of Proposition \ref{elpr}, there is an $r_0\in (1,r^*]$ so that for all $r\in (1,r_0)$ we have
$M:X^r\to X^r$. Next, by Theorem \ref{mr}, the function $\f_X$ is unbounded. Hence, $\f_{X^r}$ is unbounded as well.
Therefore, applying Theorem \ref{cor1}, we have  $M:(X^r)'\to (X^r)'$. Next, the implication $({\rm ii})\Rightarrow ({\rm iii})$ is trivial.

Turn to $({\rm iii})\Rightarrow ({\rm i})$. By Theorem \ref{mr}, $\f_{X^r}$ is unbounded. Hence, $\f_X$ is unbounded as well. Further, by the third part of Proposition~\ref{elpr},
$M:(X^r)'\to (X^r)'$ implies $M:X\to X$. It remains to apply Theorem~\ref{cor1} in order to conclude that $M:X'\to X'$.
\end{proof}

\end{document}